\documentclass[a4paper, 10pt]{article}
\usepackage{hyperref}
\usepackage{authblk}  
\usepackage{graphicx}
\usepackage{alltt}
\usepackage{amsmath}
\usepackage{amssymb}
\usepackage{amsthm}
\usepackage{authblk}
\usepackage{xcolor}
\usepackage{titling}
\usepackage{graphicx}
\usepackage{enumitem} 
\usepackage{bm}         
\usepackage{amsmath, amssymb, mathdots}

\newcommand{\rddots}{\rotatebox[origin=c]{45}{$\cdots$}}

\newtheorem{theorem}{Theorem}[section]  
\newtheorem{lemma}[theorem]{Lemma}      
\newtheorem{prop}[theorem]{Proposition}
\newtheorem{coro}[theorem]{Corollary}
\newtheorem{conj}[theorem]{Conjecture}
\newtheorem{prob}[theorem]{Problem}
\newtheorem{deff}[theorem]{Definition}
\newcommand{\nddots}{\rotatebox[origin=c]{15}{$\ddots$}}

\title{\LARGE
\textbf{Regular 3-polytopes of type $\{n,n\}$}}

\author{
  \parbox[t]{0.35\textwidth}{
    \large\textbf{Mingchao Li}\textsuperscript{1} \\
    \small\href{mailto:mli@sjtu.edu.cn}{lmc3051@sjtu.edu.cn}
  }
  \hspace{0.15em}
  \parbox[t]{0.35\textwidth}{
    \large\textbf{Wei-Juan Zhang}\textsuperscript{2} \\
    \small\href{mailto:weijuanzhang@mail.xjtu.edu.cn}{weijuanzhang@mail.xjtu.edu.cn}
  }
}

\affil{
  \small\textsuperscript{1}School of Mathematical Sciences, Shanghai Jiao Tong University, Shanghai 200240, China \\
  \small\textsuperscript{2}School of Mathematics and Statistics, Xi'an Jiaotong University, Xi'an 710049, China
}
\date{}   
\begin{document}

\maketitle

\begin{abstract}

For each integer \( n \geq 3 \), we construct a self-dual regular 3-polytope \( \mathcal{P} \) of type \( \{n, n\} \) with \( 2^n n \) flags, resolving two foundamental open questions on the existence of regular polytopes with certain Schläfli types. The automorphism group \( \operatorname{Aut}(\mathcal{P}) \) is explicitly realized as the semidirect product \( \mathbb{F}_2^{n-1} \rtimes D_{2n} \), where \( D_{2n} \) is the dihedral group of order \( 2n \), with a complete presentation for \( \operatorname{Aut}(\mathcal{P}) \) is provided. This advances the systematic construction of regular polytopes with prescribed symmetries.  

 \vskip 2mm

\textbf{Keywords:} Regular 3-polytopes, string C-groups, automorphism groups.

\end{abstract}

\section{Introduction}
\label{s:1}

Regular polytopes, as highly symmetric geometric structures, have fascinated mathematicians since ancient times. Classical examples such as the Platonic solids in three dimensions and their higher-dimensional generalizations (e.g., the  24-cell and hypercubes) epitomize the interplay between geometry, symmetry, and combinatorial elegance. In the late 20th century, the theory of abstract regular polytopes emerged as a unifying framework to analyze their combinatorial and algebraic essence.  

Formally, abstract regular polytopes are defined as ranked partially ordered sets satisfying the diamond condition and strong connectivity of flags which are their maximal chains, with automorphism groups acting transitively on their flags. This axiomatization, rooted in the work of Tits and later refined by McMullen and Schulte, allows for the classification of regular polytopes through group-theoretic structures, particularly via quotients of Coxeter groups. Such an approach bridges discrete geometry, combinatorics and group theory.  

McMullen and Schulte started constructing regular polytopes in \cite{MS90} and gave a systematic research on them in \cite{MS}. For any given regular polytope, Pellicer has constructed its regular extensions with any even number at first entry of the Schl\"{a}fli symbol (see \cite{P09}). Gabe introduced the `mixing' method and used it to construct  regular polytopes in \cite{Cu12}. 
A regular \(d\)-polytope \(\mathcal{P}\) of type \(\{p_1, \ldots, p_{d-1}\}\) is called \textit{tight} if it achieves the minimal flag count \(2p_1p_2\cdots p_{d-1}\) (see \cite{SmallestRegPolys}). Conder determined the smallest regular polytopes for each rank in \cite{SmallestRegPolys}, establishing that the minimal examples for ranks \( d \geq 9 \) are a family of tight regular polytopes of type \(\{4,\ldots,4\}\). He subsequently gave an atlas of regular polytopes with fewer than 4000 flags in \cite{RegPolys4000}. Moreover, Conder and Gabe in \cite{CC} proved that there exists a tight orientably-regular polytope of type $\{p_1,\cdots,p_{n-1}\}$ if and only if $p_i$ is an even divisor of $p_{i-1}p_{i+1}$ when defined. Leemans (with coauthors) has contributed much to the construction and classification of regular polytopes with almost simple automorphism groups. In particular, Leemans and Mulpas determined all regular polytopes with the Suzuki, Rudvalis and O'Nan sporadic groups as automorphism groups in \cite{LM}, by developing several new and powerful algorithms  for searching string C-group representations of groups.   

Among regular polytopes that have been found, it is an interesting observation that polytopes whose Schl\"{a}fli types consist entirely of odd numbers seem to be rarely discovered. A natural and fundamental inquiry arises: 
\begin{prob}\label{prob1}
Do regular 3-polytopes of type $\{m, m\}$ exist for arbitrary odd $m$? 
\end{prob}
While this serves as the starting point for our work, an additional motivation emerges from a sequence of relevant studies below. 
\medskip


In 2006, Schulte and Weiss \cite{SW06} proposed the question of characterizing groups of order $2^n$ or $2^np$ (for odd primes $p$) as automorphism groups of regular polytopes. For 2-groups, as mentioned, Conder \cite{SmallestRegPolys} constructed a tight regular polytope of type $\{4,\cdots,4\}$ for each rank $d\ge 2$. Subsequent advances \cite{HFL19, HFL20} by Hou, Feng and Leemans established the existence of regular $d$-polytopes with types $\{2^{k_1},\cdots,2^{k_{d-1}}\}$ and automorphism groups of order $2^n$, provided $d\ge 3,n\ge 5,$ $k_1,\cdots,k_{d-1}\ge 2$ and $k_1+\cdots+k_{d-1}\le n-1$.  

For the latter case, recently in \cite{HFL24}, Hou et al. showed that if a regular 3-polytope of order $2^np$ has type $\{k_1,k_2\}$, then $p$ divides $k_1$ or $k_2$. This divides possible types into two classes up to duality, 
namely Type (1) when $k_1=2^sp$ and $k_2=2^t$, and Type (2) for $k_1=2^sp$ and $k_2=2^tp$. 
While Type (1) was addressed in \cite{HFL24}---two certain families of regular 3-polytopes of Type (1) were constructed, Type (2) posed more challenges. The authors succeeded only for \( p = 3 \), producing a family of type \( \{6, 6\} \) polytopes with \( 3 \cdot 2^n \) flags (\( n \geq 5 \)). Notably, no Type (2) examples were found in the literature for primes \( p \geq 5 \), leaving the following question open:  
\begin{prob}\label{prob2}
Do regular 3-polytopes of types $\{2^sp, 2^tp\}$ with $2^np$ flags exist for primes $p \geq 5$ and some integers $s,\, t$ and $n$?
\end{prob}

In this paper, we resolve these questions by giving the following result.
\begin{theorem}\label{thm:main}
For each positive integer $n\ge 3$, there exists a self-dual regular 3-polytope $\mathcal{P}$ of Schl\"{a}fli type $\{n,n\}$ with $2^nn$ flags.
\end{theorem}

This result advances the field in three key directions:  
\begin{enumerate}
\item [{\rm (1)}] Theorem \ref{thm:main} resolves Problem \ref{prob1} affirmatively by constructing, for every odd integer \( m \geq 3 \), a regular 3-polytope with Schläfli type \( \{m, m\} \). \\[-15pt]  
 \item [{\rm (2)}] For primes $p \geq 2$ and $s = t\ge 0$, Theorem \ref{thm:main} yields regular polytopes of type $\{2^sp, 2^sp\}$ with $2^{(2^sp+s)}p$ flags, which are polytopes of Type (2) when $p\ge 3$. This addresses Problem \ref{prob2}. \\[-15pt]
  \item [{\rm (3)}] The construction applies equally to composite or even $n$, subsuming classical cases (e.g., the tetrahedron $\{3,3\}$ when $n=3$) and revealing polytopes like $\{4,4\}$ with $64$ flags, unifying Schläfli types with purely odd and even entries.
\end{enumerate}


Our theoretical framework is developed in Section~\ref{s:2}, which provides essential background on regular polytopes and their associated group-theoretic tools. Section~\ref{s:3} then introduces our novel construction for type $\{n,n\}$ polytopes, followed by the explicit automorphism group presentation.

\section{Further background}\label{s:2}
\subsection{Regular polytopes and string C-groups}

Abstract polytopes are combinatorial structures that generalize the classical notion of convex polytopes, such as the Platonic solids, by axiomatizing their incidence relations independently of geometric realization. Formally, 
\begin{deff}  
An \textit{abstract polytope} $\mathcal{P}$ of rank $d$ is a partially ordered set (poset) equipped with a rank function $\mathrm{rank}: \mathcal{P} \to \{-1,0,\ldots,d\}$, satisfying the following axioms:  
\begin{enumerate}[label=(\roman*)]  
    \item $\mathcal{P}$ contains a unique minimal element $F_{-1}$ (the \textit{empty face}) and a unique maximal element $F_d$ (the \textit{improper face}). \\[-15pt] 
    \item Every maximal totally ordered subset (\textit{flag}) contains exactly $d+2$ elements, including $F_{-1}$ and $F_d$.  \\[-15pt]  
    \item (\textit{Diamond condition}) For any two elements $F, G \in \mathcal{P}$ with $\mathrm{rank}(G) - \mathrm{rank}(F) = 2$, there exist exactly two elements $H_1, H_2$ such that $F \leq H_i \leq G$.   \\[-15pt] 
    \item (\textit{Strong connectivity}) Any two flags $\Phi, \Psi$ can be connected by a sequence of flags $\Phi = \Phi_0, \Phi_1, \ldots, \Phi_k = \Psi$, where consecutive flags differ by exactly one element.  
\end{enumerate}  
\end{deff}

Elements of rank $i$ in $\mathcal{P}$ are called \textit{$i$-faces}. An \textit{automorphism} of $\mathcal{P}$ is an order-preserving bijection from $\mathcal{P}$ to itself, mapping any $i$-face to an $i$-face and preserving the partial order structure. The collection of all automorphisms forms the automorphism group $\mathrm{Aut}(\mathcal{P})$.
Due to the combinatorial axioms of strong connectedness and the diamond condition, the stabilizer subgroup of any flag in \(\mathrm{Aut}(\mathcal{P})\) must be trivial. Consequently, every automorphism is uniquely determined by its action on a chosen flag, leading to the cardinality constraint \(|\mathrm{Aut}(\mathcal{P})| \leq f\), where \(f\) denotes the total number of flags. When equality holds, then \(\mathrm{Aut}(\mathcal{P})\) acts transitively on the flags, and we call $\mathcal {P}$ a \textit{regular} polytope. It achieves perfect alignment between the polytope's combinatorial structure and the algebraic properties of its automorphism group.



\begin{deff}  
A \textit{string group generated by involutions} (sggi) is a pair $(G, \{\rho_0, \ldots, \rho_{d-1}\})$, where $G$ is a finite group generated by involutions $\rho_i$ ($\rho_i^2 = 1$) satisfying:  
\begin{enumerate}[label=(\roman*)]  
    \item Non-adjacent generators commute: $(\rho_i \rho_j)^2 = 1$ for $|i - j| > 1$.   \\[-15pt] 
    \item Adjacent generators satisfy $(\rho_{i-1} \rho_i)^{p_i} = 1$ for some $p_i \geq 2$.  
\end{enumerate}  
The relations given above are precisely the defining relations for the Coxeter
group $[\,p_1,\, p_2,\, \dots,\, p_{d-1}\,]$, with string diagram.
If $G$ additionally satisfies the \textit{intersection condition}:   \\[-5pt] 
\[
\left\langle \rho_i \mid i \in I \right\rangle \cap \left\langle \rho_j \mid j \in J \right\rangle = \left\langle \rho_k \mid k \in I \cap J \right\rangle \quad \text{for all } I, J \subseteq \{0, \ldots, d-1\},  
\]   
then $(G, \{\rho_0, \ldots, \rho_{d-1}\})$ is called a \textit{string C-group}.  
\end{deff}  

The equivalence between regular polytopes and string C-groups is foundational:  

\begin{theorem}[{\cite{MS}}]  
There is a one-to-one correspondence between
regular $d$-polytopes $\mathcal{P}$ and string C-groups $(G, \{\rho_0, \ldots, \rho_{d-1}\})$ of rank $d$.  
\end{theorem}  

For a $d$-regular polytope $\mathcal{P}$, we fix a base flag $\Phi$. The automorphism group $G = \mathrm{Aut}(\mathcal{P})$ is generated by involutions $\rho_0,\, \rho_1,\, \dots,\, \rho_{d-1}$, where $\rho_i$ maps $\Phi$ to the unique flag differing at the $i$-face from $\Phi$. These generators form a string C-group, with the intersection condition reflecting the polytope’s face lattice.  

Conversely, given a string C-group \((G, \{\rho_0, \ldots, \rho_{d-1}\})\), a regular \(d\)-polytope can be constructed as follows: the poset is defined by setting the empty face as the unique minimal element (rank \(-1\)) and the group \(G\) itself as the maximal element (rank \(d\)). For each intermediate rank \(i\) (\(0 \leq i \leq d-1\)), the \(i\)-faces correspond to the right cosets of the subgroup \(G_i = \langle \rho_j \mid j \neq i \rangle\). Incidence between two faces is determined by non-empty intersection of their corresponding cosets, ensuring the poset satisfies the combinatorial axioms of an abstract polytope. Crucially, the group \(G\) acts transitively on the flags of this structure via right multiplication on cosets, thereby realizing the polytope’s regularity and confirming that its automorphism group is isomorphic to \(G\). 

The above group-theoretic perspective enables the study of regular polytopes through algebraic methods, revealing new families and symmetries.  

Let \(\mathcal{P}\) be a regular \(d\)-polytope with base flag \(\Phi\) and distinguished generators \(\rho_0, \ldots, \rho_{d-1}\). Its \textit{dua}l \(\mathcal{P}^*\), obtained by reversing the partial order of \(\mathcal{P}\), has distinguished generators \(\rho_{d-1}, \ldots, \rho_0\).
We say that $\mathcal{P}$ is self-dual if $\mathcal{P}$ is
isomorphic to its dual $\mathcal{P}^*$. In this case, there exists an incidence reversing bijection of $\mathcal{P}$
onto itself. The following proposition characterizes self-duality algebraically (\cite{MS}):  

\begin{prop}\label{prop:selfdual}  
Let $\mathcal{P}$ be the regular d-polytope associated
with the string C-group $G=\langle \rho_0,\, \rho_1,\, \dots,\, \rho_{d-1}\rangle$. Then $\mathcal{P}$ is self-dual if and only if there exists an
involutory group automorphism $\delta$ of $G$ such that $\delta(\rho_i)=\rho_{d-1-i} $ for each $i=0,\, 1,\, \dots,\, d-1$.

\end{prop}

\subsection{Group theory}

%

The study of group presentations often requires practical tools to estimate group orders. The following proposition establishes a computational criterion to bound the order of a group by analyzing its action on a set of cosets.  

\begin{prop}\label{prop:upperbound}  
Let \( G \) be a group and \( N \leq G \) a subgroup with \( |N| \leq n \). Suppose there exists a set \( S = \{Nh_0, Nh_1, \ldots, Nh_{m-1}\} \) of distinct right cosets of \( N \) in \( G \). If every \( g \in G \) permutes the cosets in \( S \) under right multiplication—that is, for each \( Nh_i \in S \), the product \( Nh_i g \) also lies in \( S \)—then the order of \( G \) satisfies \( |G| \leq mn \).  
\end{prop}  

\begin{proof}  
The closure of \( S \) under the action of right multiplication of \( G \) implies that \( G \) acts on \( S \) by permutations. This action decomposes \( S \) into disjoint orbits \( \mathcal{O}_1, \mathcal{O}_2, \ldots, \mathcal{O}_k \) for some $k\ge 1$, where each orbit \( \mathcal{O}_j \) has size \( m_j \) and corresponds to a stabilizer subgroup \( \mathrm{Stab}(Nh_{i_j}) \) for some coset \( Nh_{i_j} \in \mathcal{O}_j \).  

By the orbit-stabilizer theorem, for each orbit \( \mathcal{O}_j \):  
\[
|G| = |\mathcal{O}_j| \cdot |\mathrm{Stab}(Nh_{i_j})| = m_j \cdot |\mathrm{Stab}(Nh_{i_j})|.
\]  
The stabilizer \( \mathrm{Stab}(Nh_{i_j}) \) consists of elements \( g \in G \) satisfying \( Nh_{i_j} g = Nh_{i_j} \). This occurs if and only if \( h_{i_j} g \in Nh_{i_j} \), which implies \( g \in h_{i_j}^{-1} N h_{i_j} \). Thus, \( \mathrm{Stab}(Nh_{i_j}) \) is conjugate to \( N \), and hence \( |\mathrm{Stab}(Nh_{i_j})| = |N| \leq n \).  

Then we have 
\[
|G| = \sum_{j=1}^k |\mathcal{O}_j| \cdot |\mathrm{Stab}(Nh_{i_j})| \leq \sum_{j=1}^k m_j \cdot n = n \cdot \sum_{j=1}^k m_j = n \cdot m,
\]  
since the total number of cosets in \( S \) is \( m \). Therefore, \( |G| \leq mn \).  
\end{proof}  

Since any two right cosets of $N$ are either identical or disjoint, when verifying whether the coset $Nh_ig$ belongs to the set $S$, 
it suffices to choose an arbitrary element $nh_ig\in Nh_ig$ (where $n \in N$) and check whether $nh_ig\in Nh_j$ for some $j$.

\section{Proof of the main theorem}\label{s:3}
%
\subsection{A semidirect product construction}

We begin by stating three key lemmas required for our main construction.

\begin{lemma}
\label{lem1}
    Let $A \in \mathrm{GL}(n-1, \mathbb{F}_2)$ be the circulant matrix
        \[
    A = \begin{bmatrix}
    0 & 1 & 0 & \cdots & 0&0 \\
    0 & 0 & 1 & \cdots & 0&0 \\
    \vdots & \vdots & \vdots && \vdots&\vdots \\
      0 & 0 & 0 & \cdots &0& 1 \\
    1 & 1 & 1 & \cdots & 1&1 
    \end{bmatrix}_{(n-1)\times (n-1)}.
    \]
    For $1 \leq k \leq n-1$, its $k$-th power has the block structure
    \[
    A^k = \begin{bmatrix}
    \mathbf{0}_{(n-1-k)\times k} & I_{n-1-k} \\
    \mathbf{1}_{1\times (n-1)} & \\
    I_{k-1} & \mathbf{0}_{(k-1)\times (n-k)}
    \end{bmatrix},
    \]
    where $\mathbf{0}$ denotes zero matrices, $I_m$ identity matrices, and $\mathbf{1}_{1\times (n-1)} $ a row vector of ones.
\end{lemma}

\begin{proof}
    The base case $k=1$ is immediate. For $k > 1$, we compute inductively. Assume the conclusion is true for $ k-1$, then 
    \begin{align*}
    A^k &= A^{k-1}A \\
    &= \begin{bmatrix}
    \mathbf{0}_{(n-k)\times (k-1)} & I_{n-k} \\
    \mathbf{1}_{1\times (n-1)} & \\
    I_{k-2} & \mathbf{0}_{(k-2)\times (n-k+1)}
    \end{bmatrix}
    \begin{bmatrix}
    0 & 1 & \cdots & 0 \\
    \vdots & \vdots & \raisebox{0.9ex}[0pt][5pt]{$\nddots$} & 1 \\
    1 & 1 & \cdots & 1
    \end{bmatrix}_{(n-1)\times (n-1)} \\
    &= \begin{bmatrix}
    \mathbf{0}_{(n-1-k)\times k} & I_{n-1-k} \\
    \mathbf{1}_{1\times (n-1)} & \\
    I_{k-1} & \mathbf{0}_{(k-1)\times (n-k)}
    \end{bmatrix}. \qedhere
    \end{align*}
\end{proof}

\begin{lemma}
\label{lem2}
    Let $A$ be as in Lemma~\ref{lem1} and $\mathbf{v} = (1,0,\dotsc,0)^\top \in \mathbb{F}_2^{n-1}$. Then the set
    $\mathcal{B} = \left\{\mathbf{v},\, \mathbf{v}+A\mathbf{v},\, \dotsc,\, \mathbf{v}+A\mathbf{v}+\dotsb+A^{n-2}\mathbf{v}\right\}$
    forms a basis for $\mathbb{F}_2^{n-1}$. 
\end{lemma}

\begin{proof}
    Observe that $\mathcal{B}$ is linearly equivalent to $\{\mathbf{v}, A\mathbf{v}, \dotsc, A^{n-2}\mathbf{v}\}$. Let
    \[
    B =(\mathbf{v},A\mathbf{v},\cdots,A^{n-2}\mathbf{v})= \begin{bmatrix}
    1 & 0 & 0 & \cdots & 0 \\
    0 & 0 & 0 & \cdots & 1 \\
    0 & 0 & 0 & \rddots & 1 \\
    \vdots & \vdots & \rddots & \rddots & \vdots \\
    0 & 1 & 1 & \cdots & 0
    \end{bmatrix}
    \]
    be the matrix whose columns are these vectors. The anti-diagonal structure (shown via inverted diagonal dots $\rddots$) ensures $\det B \equiv 1$ $\rm {(mod 2)}$, hence $B$ is invertible over $\mathbb{F}_2$, which guarantees that  $\mathcal{B}$ is a basis for $ \mathbb{F}_2^{n-1} $.
\end{proof}

\begin{lemma}
\label{lem3}
    The matrices
    \[
    U = \begin{bmatrix}
    0 & \cdots & 1& 0 \\
    \vdots & \rddots &  \raisebox{0.5ex}[0pt][0pt]{$\rddots$} & \vdots \\
    1 & 0 & \cdots & 0 \\
    1 & 1 & \cdots & 1 
    \end{bmatrix}, \quad
    V = \begin{bmatrix}
    0 & \cdots & 0 & 1 \\
    \vdots & \rddots & 1 & 0 \\
    0 & \raisebox{0.5ex}[0pt][0pt]{$\rddots$} & \rddots & \vdots \\
    1 & 0 & \cdots & 0 
    \end{bmatrix}
    \]
    generate a dihedral group $D_{2n}$ in $\mathrm{GL}(n-1, \mathbb{F}_2)$.
\end{lemma}

\begin{proof}
    Direct verification shows $U^2 = V^2 = I$. Since $UV = A$, from Lemma~\ref{lem1}, we have $(UV)^k \neq I$ for $1 \leq k \leq n-1$, and 
        \[
    A^{n-1} = \begin{bmatrix}
    \mathbf{1}_{1\times (n-1)} & \\
    I_{n-2} & \mathbf{0}_{(n-2)\times 1}
    \end{bmatrix}.
    \]
    
While
\begin{align*}
(UV)^n = A^n &= A^{n-1}A \\
&= \begin{bmatrix}
1 & \cdots & 1 & 1 \\
1 & 0 & \cdots & 0 \\
\vdots & \ddots & \vdots & \vdots \\
0 & \cdots & 1 & 0 
\end{bmatrix}
\begin{bmatrix}
0 & 1 & 0 & \cdots & 0 \\
0 & 0 & 1 & \cdots & 0 \\
\vdots & \vdots & \vdots & \raisebox{0.9ex}[0pt][6pt]{$\nddots$} & 1 \\
1 & 1 & 1 & \cdots & 1 
\end{bmatrix} 
= \begin{bmatrix}
1 & 0 & \cdots & 0 \\
0 & 1 & \cdots & 0 \\
\vdots & \vdots & \ddots & \vdots \\
0 & 0 & \cdots & 1 
\end{bmatrix} = I_{n-1},
\end{align*}
the group $\langle U, V \rangle$ therefore has presentation
    \[
    \langle U, V \mid U^2 = V^2 = (UV)^n = I \rangle \cong D_{2n}. \qedhere
    \]
\end{proof}
\medskip

We are now ready for our semi-direct product construction.
Let \( D_{2n} = \langle h_0, h_1 \mid (h_0 h_1)^n = h_0^2 = h_1^2 = e \rangle \) be the dihedral group of order \( 2n \). We define a faithful representation
\[
\varphi : D_{2n} \to \mathrm{Aut}(\mathbb{F}_2^{n-1}) \cong \mathrm{GL}(n-1, \mathbb{F}_2)
\]  
by mapping generators to the matrices \( U \) and \( V \) from Lemma~\ref{lem3} with
$\varphi(h_0) = U$ and $\varphi(h_1) = V$.
By Lemma~\ref{lem3}, this induces an isomorphism \( D_{2n} \cong \mathrm{Im}(\varphi) \), embedding the dihedral group into the automorphism group of \( \mathbb{F}_2^{n-1} \). Further, the semidirect product 
\[
G = \mathbb{F}_2^{n-1} \rtimes_\varphi D_{2n}
\] 
has group operation
\[
(\mathbf{a}, x) \cdot (\mathbf{b}, y) = \left( \mathbf{a} + \varphi(x)^{-1}\mathbf{b},\, xy \right), \quad \forall \mathbf{a}, \mathbf{b} \in \mathbb{F}_2^{n-1},\ x, y \in D_{2n}.
\]

Within this framework, we define three canonical elements:  
\[
\begin{aligned}
\rho_0 &= \big( \mathbf{0},\, h_0 \big), \\
\rho_1 &= \big( \mathbf{0},\, h_1 \big), \\
\rho_2 &= \big( (0,\dots,0,1)^\top,\, h_0 \big),
\end{aligned}
\]  
where \( \mathbf{0} \) denotes the zero vector in \( \mathbb{F}_2^{n-1} \). These elements will serve as the generators for our string C-group structure.

\subsection{String C-group proof}

%
We demonstrate that \((G, \{\rho_0, \rho_1, \rho_2\})\) is a string C-group of type \([n, n]\), corresponding to a regular 3-polytope with Schl\"{a}fli type \(\{n, n\}\). The argument proceeds through three essential verifications: generator orders, group generation, and the intersection condition. 

First, we confirm the orders of the generators. For \(\rho_0\) and \(\rho_1\), their zero vector components remain invariant under group operation, thus their orders directly inherit from the dihedral group structure with \(o(\rho_0) = o(\rho_1) = 2\) and \(o(\rho_0\rho_1) = n\). 
We let \(\mathbf{u} = (0,\dotsc,0,1)^\top\), then the order of \(\rho_2\) follows from the computation
$\rho_2^2=\left ( \mathbf{u}+U\mathbf{u},\,h_0^2\right ) =\left ( \mathbf{0},\,e \right )$, 
where \(U\) denotes the matrix from Lemma~\ref{lem3}. This establishes \(o(\rho_2) = 2\). The commutativity \(\rho_0\rho_2 = \rho_2\rho_0\) arises from the calculation
$\rho_0\rho_2=\left (\mathbf{0}+U \mathbf{u},\, h_0^2\right ) =\left (  \mathbf{u},\, e \right ) \notag$, 
confirming \(o(\rho_0\rho_2) = 2\). 

It remains to consider the critical case of \(\rho_1\rho_2\). 
We recall $\varphi(h_1h_0)^{-1}=\varphi(h_0h_1)=UV=A$.
The recursive structure of \( (\rho_1\rho_2)^k \) emerges as
\[
(\rho_1\rho_2)^k = \left( \sum_{i=0}^{k-1} A^i\mathbf{v},\, (h_1h_0)^k \right), \quad 1 \leq k \leq n,
\]
due to the fact that $\rho_1\rho_2=\left (  \mathbf{0}+V \mathbf{u},\, h_1h_0\right ) =\left ( \mathbf{v},\, h_1h_0  \right ) $ and
\begin{align}
( \rho_1\rho_2)^k&=(\rho_1\rho_2)\cdot(\rho_1\rho_2)^{k-1}\notag \\
&=(\mathbf{v},\, h_1h_0 )\cdot \left ( \mathbf{v}+A\mathbf{v}+\cdots + A^{k-2}\mathbf{v},\, (h_1h_0)^{k-1}  \right )\notag \\
&=\left ( \mathbf{v}+A(\mathbf{v}+A\mathbf{v}+\cdots + A^{k-2}\mathbf{v}),\, (h_1h_0)^k \right )  \notag \\
&=\left ( \mathbf{v}+A\mathbf{v}+\cdots + A^{k-1}\mathbf{v},\, (h_1h_0)^k \right ) \notag.
\end{align}
For \( k < n \), the non-trivial group component \( (h_1h_0)^k \neq e \) prevents identity. At \( k = n \), the characteristic polynomial
\[
\det(\lambda I - A) = \begin{vmatrix}
\lambda & 1 & 0 & \cdots & 0 \\
0 & \lambda & 1 & \ddots & \vdots \\
\vdots & \ddots & \ddots & \ddots & 0 \\
0 & \cdots & 0 & \lambda & 1 \\
1 & 1 & \cdots & 1 & \lambda+1
\end{vmatrix}= \lambda^{n-1} + \cdots + \lambda + 1 \in \mathbb{F}_2[\lambda]
\]
induces \( \sum_{i=0}^{n-1} A^i = \mathbf{0} \) via the Hamilton-Cayley theorem. Consequently,
 \[
    (\rho_1\rho_2)^n = \left( \sum_{i=0}^{n-1}A^i\mathbf{v},\ e \right) = (\mathbf{0},\ e). 
 \]

Next, we verify $G=\left \langle \rho_0,\rho_1,\rho_2  \right \rangle $, which means the elements specified indeed form a generating set for $G$. For \( 1 \leq k \leq n-1 \), consider the composite elements
\[
(\rho_1\rho_2)^k(\rho_0\rho_1)^k = \left( \sum_{i=0}^{k-1} A^i\mathbf{v},\ e \right),
\]
where the vanishing second component results from \((h_1h_0)^k(h_0h_1)^k = e\). By Lemma~\ref{lem2}, these elements generate the entire vector space \(\mathbb{F}_2^{n-1}\) through their first components. Moreover, it is clear that \(\langle \rho_0, \rho_1 \rangle\) faithfully reproduces \(D_{2n}\) via pure group elements. The semidirect product structure therefore guarantees \( G = \langle \rho_0, \rho_1, \rho_2 \rangle \).

Clearly, both of $\langle \rho_0 \rangle \cap \langle \rho_1\rangle$ and $\langle \rho_1 \rangle \cap \langle \rho_2\rangle$ are trivial. To complete the string C-group verification, it suffices to establish the intersection condition $\langle \rho_0, \rho_1 \rangle \cap \langle \rho_1, \rho_2 \rangle = \langle \rho_1 \rangle$. The subgroup $\langle \rho_0, \rho_1 \rangle$ consists entirely of elements with trivial vector components $(\mathbf{0}, g)$, where $g$ ranges over $D_{2n}$. Specifically, these elements are rotational operations $(\rho_0\rho_1)^k = (\mathbf{0}, (h_0h_1)^k)$ and their cosets $(\rho_0\rho_1)^k\rho_0 = (\mathbf{0}, (h_0h_1)^kh_0)$ for $0 \leq k \leq n-1$, maintaining strict zero vector preservation under the semidirect action. 

The subgroup \(\langle \rho_1, \rho_2 \rangle\) exhibits complete algebraic closure with elements spanning four categories: 
rotational elements \((\rho_1\rho_2)^k = \left( \sum_{i=0}^{k-1}A^i\mathbf{v}, (h_1h_0)^k \right)\) and
reflective elements \((\rho_1\rho_2)^k\rho_1 = \left( \sum_{i=0}^{k-1}A^i\mathbf{v}, (h_1h_0)^kh_1 \right)\) for \(1 \leq k \leq n-1\),
the identity element $(\mathbf{0}, e) = (\rho_1\rho_2)^0$ and
the generator $\rho_1 = (\mathbf{0}, h_1)=(\rho_1\rho_2)^0\rho_1$.
However, the vectors $\sum_{i=0}^{k-1}A^i\mathbf{v}$ form a basis of $\mathbb{F}_2^{n-1}$ by Lemma~\ref{lem2}, ensuring $\sum_{i=0}^{k-1}A^i\mathbf{v} \neq \mathbf{0}$ when $1 \leq k \leq n-1$. 

For an element to belong to both $\langle \rho_0, \rho_1 \rangle$ and $\langle \rho_1, \rho_2 \rangle$, it must possess a trivial vector component $\mathbf{0}$ inherited from $\langle \rho_0, \rho_1 \rangle$, and 
have its group component $g \in \langle h_1 \rangle$ to maintain compatibility with $\langle \rho_1, \rho_2 \rangle$. 
This excludes all rotational or reflective elements with $k \geq 1$, leaving only
$\{ (\mathbf{0}, e), (\mathbf{0}, h_1) \} = \langle \rho_1 \rangle.$
Thus we derive $\langle \rho_0, \rho_1 \rangle \cap \langle \rho_1, \rho_2 \rangle = \langle \rho_1 \rangle$, satisfying the string C-group intersection criterion.
\medskip

%

This completes the constructive proof of Theorem~\ref{thm:main}. To fully characterize \(G\) as the automorphism group of the associated regular polytope, say \(\mathcal{P}\), and subsequently prove the self-duality of $\mathcal{P}$, we will derive an explicit presentation for \(G\) in terms of generators \(\rho_0, \rho_1, \rho_2\) and their defining relations in the consequent subsection. 

\subsection{Presentation for \(\mathrm{Aut}(\mathcal{P})\)}\label{s:4}

In this subsection we demonstrate the presentation for $G$. Specifically, it suffices to prove the following theorem:
\begin{theorem}
    For any integer $n\ge 3$, the group $(G,\{\rho_0,\rho_1,\rho_2\})$ admits the presentation 
\vspace{-6pt}
   \[
G = \left\langle \rho_0, \rho_1, \rho_2 \,\left|\, 
\begin{array}{l}
\rho_0^2=\rho_1^2=\rho_2^2=(\rho_0\rho_2)^2
=(\rho_0\rho_1)^n=(\rho_1\rho_2)^n=1,\\
\left( (\rho_1\rho_2)^k(\rho_0\rho_1)^k \right)^2=1~\text{for } 1 \leq k \leq n-1
\end{array} \right.\right\rangle.
\] 
\end{theorem}

\begin{proof}
For any integer \( n \ge 3 \), we establish the presentation of \( G \) by analyzing the structure of a related quotient group. Let \( H \) be the quotient of the Coxeter group \( [n, n] \) generated by \( \mu_0, \mu_1, \mu_2 \) with standard Coxeter relations:
\[
\mu_0^2 = \mu_1^2 = \mu_2^2 = (\mu_0\mu_2)^2 = (\mu_0\mu_1)^n = (\mu_1\mu_2)^n = 1,
\]
adjoined with additional relations \( \left( (\mu_1\mu_2)^k(\mu_0\mu_1)^k \right)^2 = 1 \) for \( 1 \leq k \leq n-1 \). Since the generators \( \rho_0, \rho_1, \rho_2 \) of \( G \) satisfy these relations, there exists an epimorphism \( \pi: H \to G \) mapping \( \mu_i \mapsto \rho_i \). The orders of elements in \( H \) are preserved:
\[
o(\mu_0) = o(\mu_1) = o(\mu_2) = o(\mu_0\mu_2) = 2, \quad o(\mu_0\mu_1) = o(\mu_1\mu_2) = n,
\]
implying \( |H| \geq |G| = 2^n n \).

To bound \( |H| \), consider the subgroup \( N = \left\langle (\mu_1\mu_2)^k(\mu_0\mu_1)^k \mid 1 \leq k \leq n-1 \right\rangle \). Each generator of \( N \) has order 2 due to the added relations, and commutativity follows from 
\[
(\mu_1\mu_2)^i(\mu_0\mu_1)^i \cdot (\mu_1\mu_2)^j(\mu_0\mu_1)^j = (\mu_1\mu_0)^i(\mu_1\mu_2)^{j-i}(\mu_0\mu_1)^j = (\mu_1\mu_0)^j(\mu_1\mu_2)^{i-j}(\mu_0\mu_1)^i.
\]
Thus, \( N \) is elementary abelian of order \( 2^{n-1} \). 

Let \( D = \langle \mu_0\mu_1, \mu_0 \rangle \cong D_{2n} \), consisting elements of the identity $1,\, \mu_0,\, (\mu_0\mu_1)^k$ and $(\mu_0\mu_1)^k\mu_0$ for $1\le k\le n-1$. We then define the cosets set \( S = \{ Nh \mid h \in D \} \). 
Since $\mu_0,\mu_1\in D$, apparently we have 
$\{ Nh\mu \mid h \in D,\, \mu=\mu_0~\rm{or}~\mu_1\}\subseteq S\notag$.
It remains to verify $ \{ Nh\mu_2 \mid h \in D \}\subseteq S$.

Observing $\mu_0\mu_2=\mu_2\mu_1\mu_1\mu_0=(\mu_1\mu_2)^{n-1}(\mu_0\mu_1)^{n-1}\in N$,
one has $N\mu_2=N\mu_0\in S$ and $N\mu_0\mu_2=N\in S$. 
Further, for \( (\mu_1\mu_2)^{n-k}\mu_2 \), we show
\begin{align}
(\mu_1\mu_2)^{n-k}\mu_2&=\left ( (\mu_1\mu_2)^{n-k}(\mu_0\mu_1)^{n-k} \right )(\mu_0\mu_1)^k\mu_2\in N(\mu_0\mu_1)^k\mu_2\notag~~~and\\
(\mu_1\mu_2)^{n-k}\mu_2&
=(\mu_1\mu_2)^{n-k-1}(\mu_0\mu_1)^{n-k-1}(\mu_0\mu_1)^k\mu_0\in N(\mu_0\mu_1)^k\mu_0\notag,
\end{align}
implying 
\begin{align}
N(\mu_0\mu_1)^k\mu_2&=N(\mu_0\mu_1)^k\mu_0\in S\notag~~~and\\ 
N(\mu_0\mu_1)^k\mu_0\mu_2&=N(\mu_0\mu_1)^k\in S\notag.
\end{align}
Therefore, $ \{ Nh\mu_2 \mid h \in D \}\subseteq S$. By Proposition \ref{prop:upperbound}, $|H|\le |N|\cdot|D|=2^{n-1}\cdot2n=2^nn$.
Since \( |H| \geq |G| = 2^n n \), equality holds, and \( \pi \) induces an isomorphism between$ (H,\{\mu_0,\mu_1,\mu_2\}) $ and $(G,\{\rho_0,\rho_1,\rho_2\}) $. This completes the proof.
\end{proof}

With the presentation, it is obtainable that for each $n\ge 3$, the regular 3-polytope \(\mathcal{P}\) associated with the string C-group \((G, \{\rho_0, \rho_1, \rho_2\})\) is self-dual. In fact, this can be reflected from the symmetry of \(G\)'s presentation, where \(\rho_0\) and \(\rho_2\) are interchangeable generators.
We construct an involutory map \(\delta: G \to G\) defined by \(\delta(\rho_0) = \rho_2\), \(\delta(\rho_1) = \rho_1\), and \(\delta(\rho_2) = \rho_0\), which swaps the roles of \(\rho_0\) and \(\rho_2\) while fixing \(\rho_1\). The map \(\delta\) is an automorphism of $G$, as it preserves all defining relations in the presentation of \(G\): the Coxeter relations \((\rho_0\rho_2)^2 = 1\) and \((\rho_0\rho_1)^n = (\rho_1\rho_2)^n = 1\) are invariant under \(\delta\), while the additional relations \(\left( (\rho_1\rho_2)^k(\rho_0\rho_1)^k \right)^2 = 1\) transform into equivalent identities via inversion. By Proposition~\ref{prop:selfdual}, the existence of such a relation-preserving automorphism \(\delta\) implies that \(\mathcal{P}\) is isomorphic to its dual \(\mathcal{P}^*\). 
\medskip

Now we have completed the proof of Theorem~\ref{thm:main}, establishing the existence of the designed string C-group structure. As direct consequences of this result, we provide affirmative answers to the existence questions for certain Schläfli types posed in Section~\ref{s:1}:

\begin{coro}
\label{cor:odd}
For every odd integer $m \geq 3$, there exists a self-dual regular $3$-polytope with Schl\"{a}fli type $\{m, m\}$ having exactly $2^{m}m$ flags.
\end{coro}

\begin{coro}
\label{cor:prime-power}
For any prime number $p \geq 2$ and integer $s \geq 0$, there exists a self-dual regular $3$-polytope with Schl\"{a}fli type $\{2^{s}p, 2^{s}p\}$ possessing precisely $2^{(2^{s}p + s)}p$ flags.
\end{coro}

%
\medskip

Suggested by computational evidence from {\sc Magma}~\cite{Magma}, we propose the following structural conjecture for regular polytopes with arbitrary ranks:

\begin{conj}
For integers \( d \geq 3 \) and \( n \geq 3 \), there exists a regular \( d \)-polytope of type \( \{n, n, \dots, n\} \) whose automorphism group is a quotient of the Coxeter group \( [n, n, \dots, n] \). The generating involutions \( \rho_0, \rho_1, \dots, \rho_{d-1} \) are restricted by:\\[-12pt]
\begin{enumerate}[label=(\roman*)]  
    \item 
    \(\rho_i^2 = 1\) for all \( 0 \leq i \leq d-1 \), 
    \( (\rho_i\rho_j)^2 = 1 \) whenever \( 0 \leq i < j-1 \le d-2 \), 
    and \( (\rho_i\rho_{i+1})^n = 1 \) for \( 0 \leq i \leq d-2 \);\\[-12pt]
    
    \item 
    \(\left( (\rho_{r+1}\rho_{r+2})^k (\rho_r\rho_{r+1})^k \right)^2 = 1\) 
    for all \( 0 \leq r \leq d-3 \) and \( 1 \leq k \leq n-1 \).
\end{enumerate}
\end{conj}

Initial verification through {\sc Magma}~\cite{Magma} calculations confirms this conjecture for low ranks $d$ with small parameters $n$. Our current research focuses on characterizing the group structure, aiming to extend Theorem \ref{thm:main} to arbitrary rank \( d \geq 3 \). 
This generalization poses significant complexity and technical challenges.

\paragraph{Acknowledgment.}
The second author is supported by the Fundamental Research Funds for the Central Universities (Grant No. xxj032025053) and the National Natural Science Foundation of China (Grant No. 12201486).



\begin{thebibliography}{99}

{\small 

\bibitem{Magma}
W.\ Bosma, J.\ Cannon and C.\ Playoust, 
The Magma algebra system I: the user language, 
{\em J.\ Symbolic Comput.} 24 (1997), 235--265.  

%

\bibitem{SmallestRegPolys}
M.\,D.\,E.\ Conder, 
The smallest regular polytopes of given rank, 
{\em Adv.\ Math.} 236 (2013), 92--110.  

\bibitem{RegPolys4000}
M.\,D.\,E.\ Conder, 
`Regular polytopes with up to 4000 flags', available at \\
{\tt www.math.auckland.ac.nz/{$\sim$}conder/RegularPolytopesWithUpTo4000Flags-ByOrder.txt}.   

\bibitem{CC}
M.\,D.\,E.\ Conder and G.\ Cunningham, 
Tight orientably-regular polytopes, 
{\em Ars Math.\ Contemp.} 8 (2015), 68--81. 



%
%

%


\bibitem{Cu12}
G.\ Cunningham, 
Mixing regular convex polytopes, 
{\em Discrete Math.} 312 (2012), no. 4, 763--771. 








\bibitem{HFL19}
D.-D.\ Hou, Y.-Q.\ Feng and D.\ Leemans,  
Existence of regular 3-polytopes of order $2^n$,
{\em J.\ Group Theory\/} 22 (2019), no. 4, 579--616.

\bibitem{HFL20}
D.-D.\ Hou, Y.-Q.\ Feng and D.\ Leemans,  
On regular polytopes of 2-power order,
{\em Discrete Comput.\ Geom.} 64 (2020), no. 2, 339--346.

\bibitem{HFL24}
D.-D.\ Hou, Y.-Q.\ Feng and D.\ Leemans,  
Regular 3-polytopes of order $2^n p$, available at 
{\tt https://arxiv.org/abs/2001.02945}.


\bibitem{LM}
D.\ Leemans and J.\ Mulpas, 
The string C-group representations of the Suzuki, Rudvalis and O'Nan sporadic groups, 
{\em Art Discrete Appl. Math.}, 5 (2022), no. 3, Paper No. 3.09, 12 pp. 

\bibitem{MS90}
P.\ McMullen and E.\ Schulte, 
Constructions for regular polytopes,
{\em J.\ Combin.\ Theory Ser.\ A\/} 53 (1990), no. 1, 1--28.

\bibitem{MS}
P.\ McMullen and E.\ Schulte, 
{\em Abstract Regular Polytopes\/}, 
Encyclopedia Math.\ Appl., vol.\ 92, Cambridge University Press, 2002.   



\bibitem{P09}
D.\ Pellicer, 
Extensions of regular polytopes with pressigned Schl{\" a}fli symbol, 
{\em J.\ Combin.\ Theory Ser.\ A\/} 116 (2009) 303--313. 





%
%
%

\bibitem{SW06}
E.\ Schulte and  A.\,I.\ Weiss, Problems on polytopes, their groups, and realizations, 
{\em Period.\ Math.\ Hungar.} 53 (2006), no. 1-2, 231--255. 



}

\end{thebibliography}
\end{document}